\documentclass[reqno0,12pt]{amsart}
\usepackage{kpfonts,baskervald}
\pdfoutput=1
\usepackage{bbm}
\usepackage[nodate]{datetime}
\usepackage[hmargin=30mm,top=30mm,bottom=40mm,a4paper]{geometry}
\usepackage{color}
\usepackage{fancyhdr}
\usepackage{amsmath,amssymb,amsfonts,amsthm}
\usepackage{amstext}
\usepackage{amsmath}
\usepackage{amssymb}
\usepackage{enumerate}
\usepackage{amsbsy}
\usepackage{amsopn}
\usepackage{bbm,amsthm}
\usepackage{amscd}
\usepackage{amsxtra}
\usepackage{enumitem}
\usepackage{todonotes}

\newtheorem{theorem}{Theorem}[section]
\newtheorem*{theorem*}{Theorem}
\newtheorem{lemma}{Lemma}[section]

\theoremstyle{remark}

\newcommand{\CC}{\mathbb{C}}

\newcommand{\NN}{\mathbb{N}}

\newcommand{\ind}{\mathbbm{1}}

\begin{document}
\title{Multiplicative functions supported on the $k$-free integers with small partial sums}
\author{Marco Aymone, Caio Bueno, Kevin Medeiros}
\begin{abstract}
We provide examples of multiplicative functions $f$ supported on the $k$-free integers such that at primes $f(p)=\pm 1$ and such that the partial sums of $f$ up to $x$ are $o(x^{1/k})$. Further, if we assume the Generalized Riemann Hypothesis, then we can improve the exponent $1/k$: There are examples such that the partial sums up to $x$ are $o(x^{1/(k+\frac{1}{2})+\epsilon})$, for all $\epsilon>0$. This generalizes to the $k$-free integers the results of  Aymone, `` {\em A note on
  multiplicative functions resembling the {M}\"{o}bius function}'', J. Number
  Theory, 212 (2020), pp.~113--121. 
\end{abstract}
\maketitle

\section{Introduction.}

We say that an integer $n$ is \textit{$k$-free} if $n$ is not divisible by any prime power $p^k$, where $k\geq 2$ is an integer; we say that $f:\NN\to\CC$ is multiplicative if $f(nm)=f(n)f(m)$ for all positive integers $n$ and $m$ such that $\gcd(n,m)=1$, and we say that $f$ is completely multiplicative if this relation holds for all $n$ and $m$. Here the Vinogradov notation $f(x)\ll g(x)$ means that for some constant $C>0$, $|f(x)|\leq C|g(x)|$, for all sufficiently large $x>0$. 

In this paper we are interested in the partial sums $\sum_{n\leq x} f(n)$ where $f:\NN\to\{-1,0,1\}$ is a multiplicative function such that at primes $f(p)=\pm1$ and is supported on the $k$-free integers, \textit{i.e.}, functions $f$ such that $f(n)=0$ if $n$ is not $k$-free. 

When we consider the case of the squarefree integers (the case $k=2$), if $f$ is multiplicative, supported on the squarefree integers and $f(p)=-1$ for all primes $p$, then $f=\mu$, the M\"obius function. A well known fact is that the Riemann hypothesis (RH) is equivalent to the square-root cancellation of the partial sums of the M\"obius function $\mu$, more precisely, $\sum_{n\leq x}\mu(n)\ll x^{1/2+\epsilon}$, for all $\epsilon>0$. Up to this date, the true magnitude of the partial sums of the M\"obius function is still unknown, even conditionally on RH; the unconditional upper bound obtained by the classical zero free region for the Riemann zeta function is $\sum_{n\leq x}\mu(n)\ll x\exp(-c\sqrt{\log x})$, and the best conditional upper bound is due to Soundararajan \cite{soundararajan_mobius} which states that under RH, $\sum_{n\leq x}\mu(n)\ll \sqrt{x}\exp(\sqrt{\log x}(\log\log x)^{14})$. 
 
 These results for the M\"obius function naturally raises the question of how small the partial sums of a multiplicative function supported on the squarefree integers can be. This led Wintner \cite{wintner} to investigate a random model for the M\"obius function which consists in randomizing the values $f(p)=\pm1$ at primes $p$, and thus ending up with a random multiplicative function. Wintner proved that a random multiplicative function has partial sums $\sum_{n\leq x}f(n)\ll x^{1/2+\epsilon}$, for all $\epsilon>0$, almost surely, and also that these partial sums are not $\ll x^{1/2-\epsilon}$, almost surely. Wintner's seminal paper led to many investigations on the partial sums of a random multiplicative function. Here we refer for the interested reader: \cite{basquinn}, \cite{erdosuns}, \cite{halasz}, \cite{harpergaussian}, \cite{harpermult} and \cite{tenenbaum2013}. The best upper bound up to date is due to Basquin \cite{basquinn} and to Lau-Tenenbaum-Wu \cite{tenenbaum2013}: $\sum_{n\leq x}f(n) \ll \sqrt{x}(\log\log x)^{2+\epsilon}$, for all $\epsilon>0$, almost surely; the best $\Omega$ bound is due to a very recent result of Harper \cite{harperomegabound}: The partial sums $\sum_{n\leq x}f(n)$ are not $\ll \sqrt{x}(\log\log x)^{1/4-\epsilon}$, for all $\epsilon>0$, almost surely.
 
By these results for a random multiplicative function we can naturally ask if there are examples of multiplicative functions $f$ supported on the squarefree integers such that at primes $f(p)=\pm 1$, and such that its partial sums $\sum_{n\leq x}f(n)$ have cancellation beyond the $\sqrt{x}$-cancellation. In \cite{aymone_beyond_onehalf}, the first author exhibited examples of such functions such that $\sum_{n\leq x}f(n)\ll \sqrt{x}\exp(-c(\log x)^{1/4})$, and further, by assuming RH we can go much more beyond $\sqrt{x}$-cancellation: There are examples such that $\sum_{n\leq x}f(n)\ll x^{2/5+\epsilon}$, for all $\epsilon>0$. Here we ask the same question over the $k$-free integers: 

\noindent{\textbf{Question.}} \textit{For which values of exponents $\alpha>0$ there exists a multiplicative function $f$ that is supported on the $k$-free integers, at primes $f(p)=\pm1$ and such that $\sum_{n\leq x}f(n)\ll x^\alpha$?} 

The main reason to assume that $f(p)=\pm1$ over primes $p$, is that this guarantees that $\sum_{n\leq x}|f(n)|$ is not $o(x)$. With this condition, the event in which $f$ has small partial sums means a non-trivial cancellation among the values $f(n)$, as a consequence of the choice of the values $f(p^m)$, for primes $p$ and powers $m\geq 1$.

Our first result states:

\begin{theorem}\label{teorema 1} There exists a multiplicative function $f:\NN\to\{-1,0,1\}$ supported on the $k$-free integers such that at primes $p$, $f(p)=\pm 1$, and such that 
\begin{equation*}
\sum_{n\leq x}f(n)\ll\frac{ x^{1/k}}{\exp(c(\log x)^{1/4})},
\end{equation*}
for some constant $c>0$.
\end{theorem}

Thus, we can go beyond $x^{1/k}$-cancellation when we consider $k$-free integers, and the exponent $1/k$ is critical in our results -- If we assume the Generalized Riemann Hypothesis (GRH \footnote{We recall that $\zeta$ stands for the classical Riemann zeta function and for a Dirichlet character $\chi$, $L(s,\chi)$ stands for the classical $L$-function $L(s,\chi)=\sum_{n=1}^\infty \frac{\chi(n)}{n^s}$. The Riemann Hypothesis for $\zeta$ states that all zeros $\zeta(s_0)=0$ with real part $Re(s_0)>0$  have real part equals $1/2$. The Generalized Riemann Hypothesis is the same statement for $L(s,\chi)$.}), then we can go beyond the exponent $1/k$:

\begin{theorem}\label{teorema riemann} Assume the Generalized Riemann Hypothesis. Then there exists a multiplicative function $f:\NN\to\{-1,0,1\}$ supported on the $k$-free integers such that at primes $p$, $f(p)=\pm 1$, and such that  
$$\sum_{n\leq x}f(n)\ll x^{1/(k+\frac{1}{2})+\epsilon},$$
for any $\epsilon>0$.
\end{theorem}

Our examples in Theorems \ref{teorema 1} and \ref{teorema riemann} are constructed in light of the extreme examples in the Erd\"os discrepancy theory for multiplicative functions. 

The original formulation of the Erd\"os discrepancy problem (EDP) is whether there exists an arithmetic function $f:\NN\to\{-1,1\}$ such that its partial sums over any homegenous arithmetic progression are uniformly bounded by some constant $C>0$. As shown by Tao in \cite{taodiscrepancy}, no such function exists, and moreover, we might say that the essence of the EDP is whether a completely multiplicative function $f:\NN\to\{z\in\CC: |z|=1\}$ has bounded or unbounded partial sums, see \cite{taodiscrepancy}.

A consequence of the solution of the EDP is that a completely multiplicative function $f:\NN\to\{-1,1\}$ has unbounded partial sums. It is important to mention that if we allow completely multiplicative functions vanish over a finite subset of primes, then we might get bounded partial sums, for instance, any real and non-principal Dirichlet character $\chi$. These Dirichlet characters are what we call the extreme examples in the Erd\"os discrepancy theory for multiplicative functions.

Following Granville and Soundararajan \cite{granvillepretentious}, define the ``distance'' up to $x$ between two multiplicative functions $|f|,|g|\leq 1$ as
\begin{equation*}
\mathbb{D}(f,g;x):=\left(\sum_{p\leq x}\frac{1-Re(f(p)\overline{g(p)})}{p}\right)^{1/2},
\end{equation*}
where in the sum above $p$ denotes a generic prime. Moreover, let $\mathbb{D}(f,g):=\mathbb{D}(f,g;\infty)$, and we say that $f$ pretends to be $g$, or that $f$ is $g$-pretentious if $\mathbb{D}(f,g)<\infty$. 

In the resolution of the EDP (see Remark 3.1 of \cite{taodiscrepancy} and Proposition 2.1 of \cite{aymone_discrepancy}),  it has been shown that a multiplicative function $f:\NN\to[-1,1]$ with bounded partial sums must be $\chi$-pretentious, for some real Dirichlet character $\chi:\NN\to\{-1,0,1\}$, provided that the logarithmic average $\frac{1}{\log x}\sum_{\sqrt{x}\leq n\leq x}\frac{|f(n)|^2}{n}\gg 1$. Thus, if we seek for multiplicative functions supported on the $k$-free integers with small partial sums, then we must look firstly to those that are $\chi$-pretentious for some real and non-principal Dirichlet character $\chi$. In our particular case, to construct examples for theorems \ref{teorema 1} and \ref{teorema riemann} we assume much more than $\chi$-pretentiousness. Indeed, we deduce the Theorem \ref{teorema 1} from the following result:
\begin{theorem}\label{theorem 1} Let $\mu_k^2(n)\in\{0,1\}$ be the indicator that $n$ is $k$-free. Let $g:\NN\to\{-1,1\}$ be a completely multiplicative function that is $\chi$-pretentious for some real and non-principal Dirichlet character $\chi$, and further assume that
\begin{equation}\label{condicao 1}
\sum_{p\leq x}|1-g(p)\chi(p)|\ll \frac{x^{1/k}}{\exp(c\sqrt{\log x})}.
\end{equation}
Let $f=\mu_k^2g$ be a multiplicative function supported on the $k$-free integers. Then there exists a constant $\lambda>0$ such that
\begin{equation}\label{condicao 2}
\sum_{n\leq x}f(n)\ll \frac{x^{1/k}}{\exp(\lambda(\log x )^{1/4})}.
\end{equation}
\end{theorem}
The condition \eqref{condicao 1} above essentialy says that the completely multiplicative function $g$ is obtained by modifying a Dirichlet character $\chi$ in a quantity $\ll x^{1/k}\exp(-c\sqrt{\log x})$ of primes $p$ below $x$, and then the example for Theorem \ref{teorema 1} is obtained by setting $f=\mu_k^2g$. In contrast, to get an example as in Theorem \ref{teorema riemann}, we do not have such flexibility and we assume that our completely multiplicative function $g$ coincides with a Dirichlet character in all, but a finite subset of primes $p$, and then we set $f=\mu_k^2g$ just as in Theorem \ref{teorema 1}. Thus, these are our examples of multiplicative functions supported on the $k$-free integers with the smallest partial sums (under GRH), and are what we call real and non-principal Dirichlet characters restricted to the $k$-free integers.

We observe that if $f$ is within our class of examples of Theorem \ref{teorema riemann}, then the Dirichlet series of $f$, $F(s):=\sum_{n=1}^\infty f(n)n^{-s}$, is essentially the Dirichlet series of $\mu_k^2\chi$ which is $L(s,\chi)/\zeta(ks)$ if $k$ is even, and it is $L(s,\chi)/L(ks,\chi)$ if $k$ is odd. Thus, under GRH, the Dirichlet series of $f$ has analytic continuation to the half plane $Re(s)>1/2k$. This suggests that the partial sums up to $x$ of $f$ could have cancellation $\ll x^{1/2k+\epsilon}$, which would be much better than the result of Theorem \ref{teorema riemann}. However, it is important to note that the existence of analytic continuation is not enough to ensure this cancellation.  

Finally, if $f$ is in the class of examples of Theorem \ref{teorema riemann}, then $f$ can not have partial sums up to $x$ that are $\ll x^{1/2k-\epsilon}$, for any $\epsilon>0$. This is because 
 $\zeta(s)$ and $L(s,\chi)$ have an infinite number of zeros with $Re(s)=1/2$. Thus,  the Dirichlet series of $f$ will have an infinite number of poles at $Re(s)=1/2k$, unless that, in the case that $k$ is even, every zero of $\zeta(ks)$ corresponds to a zero of $L(s,\chi)$ in the line $Re(s)=1/2k$, or in the case that $k$ is odd, every zero of $L(ks,\chi)$ corresponds to a zero of $L(s,\chi)$ in the line $Re(s)=1/2k$. But this cannot happen due to zero density estimates, see the Grand Density Theorem (Theorem 10.4, page 260 of \cite{kowalski_livro}).

\section{Preliminaries}
In this section we state classical results that we will need in the course of the proofs of the main results. 

\subsection{Partial sums of $\mu$ and of $\mu\chi$}
As stated in the introduction, for some constant $c>0$, the partial sums of the M\"obius function are (unconditionally)
\begin{equation}\label{equacao somas parciais da mobius incondicional}
\sum_{n\leq x}\mu(n)\ll\frac{x}{\exp(c\sqrt{\log x})}.
\end{equation}

This can be obtained by adapting the proof of the Prime Number Theorem as in \cite{montgomery_livro}, chapter 6. Indeed, this follows by combining an effective Perron's formula (Corollary 5.3 of \cite{montgomery_livro}) with the classical zero free region for the Riemann zeta function, namely those points $s=\sigma+it$ such that $\sigma\geq 1-\frac{c}{\log(|t|+4)}$, for some constant $c>0$, and the estimate $1/ \zeta(s)\ll \log(|t|+4)$, valid for $s$ in the classical zero free region with $|t|\geq 1$.

If $\chi$ is a real and non-principal Dirichlet character, then we also have that for some constant $c>0$, the function $\mu\chi$ has (unconditionally) partial sums
\begin{equation}\label{equacao somas parciais da mobius chi incondicional}
\sum_{n\leq x}\mu(n)\chi(n)\ll\frac{x}{\exp(c\sqrt{\log x})}.
\end{equation}

The proof of this result follows the same line of reasoning of the proof of \eqref{equacao somas parciais da mobius incondicional} with some extra care. This extra care is related to exceptional zeros of $L$ functions. Indeed, the Dirichlet series of $\mu\chi$ is $1/L(s,\chi)$, where $L(s,\chi)$ is the Dirichlet $L$ function associated to $\chi$. The $L$ function $L(s,\chi)$ has a zero free region of the same type of Riemann zeta function with at most one possible exception, a real zero $\beta<1$ inside this region, see \cite{montgomery_livro} Theorem 11.3. Further, by Theorem 11.4 of \cite{montgomery_livro}, whether there exists an exceptional zero or not, we have that for $s=\sigma+it$ in this zero free region with sufficiently large $|t|$, $1/L(s,\chi)\ll \log(|t|+4)$, and thus the proof of \eqref{equacao somas parciais da mobius chi incondicional} is completed with an effective Perron's formula. 

Now if we assume the Riemann Hypothesis (RH), then for any $\epsilon>0$ the partial sums
\begin{equation}\label{equacao somas parciais da mobius condicional}
\sum_{n\leq x}\mu(n)\ll x^{1/2+\epsilon}.
\end{equation}

A proof of this can be found at \cite{tit} Theorem 14.25 (C). In a few words, under RH we can deduce that for $s=\sigma+it$ with $\sigma\geq \sigma_0>1/2$ and $|t|$ sufficently large, $\log \zeta(s)\ll_{\sigma_0,\epsilon} (\log(|t|+4))^{2(1-\sigma)+\epsilon}$, for any $\epsilon>0$, where $\ll_{\sigma_0, \epsilon}$ means that the implicit constant depends only in $\sigma_0$ and $\epsilon$, see Theorem 14.2 of \cite{tit}. Hence, under RH, $1/\zeta(s)$ is analytic in $Re(s)>1/2$ and $1/\zeta(s)\ll_{\sigma_0,\epsilon}|t|^\epsilon$, for any $\epsilon>0$. Then, the proof of \eqref{equacao somas parciais da mobius condicional} is completed with an effective Perron's formula.

Finally, under GRH, for a real and non-principal Dirichlet character $\chi$, for any $\epsilon>0$ we have that
\begin{equation}\label{equacao somas parciais da mobius chi condicional}
\sum_{n\leq x}\mu(n)\chi(n)\ll x^{1/2+\epsilon},
\end{equation}
and the proof follows the lines of the proof of \eqref{equacao somas parciais da mobius condicional}.
\subsection{The Dirichlet Hyperbola Method}\label{subsection dirichlet hyperbola} This method allow us to evaluate the partial sums of the Dirichlet convolution between two arithmetic functions $h$ and $g$. We recall that the Dirichlet convolution is defined as $h\ast g(n)=\sum_{d|n}h(d)g(n/d)$, where the notation $d|n$ means that $d$ divides $n$. Denote $M_f(x):=\sum_{n\leq x}f(n)$. The Dirichlet hyperbola method states: If $UV=x$ with $U,V\geq1$, then
\begin{equation}\label{equacao dirichlet hyperbola}
\sum_{n\leq x}(h\ast g)(n)=\sum_{n\leq U}h(n)M_g\bigg{(}\frac{x}{n}\bigg{)}+\sum_{n\leq V}g(n)M_h\bigg{(}\frac{x}{n}\bigg{)}-M_g(V)M_h(U).
\end{equation}

\section{Proof of the main results}

\subsection{Some words about the proof in the $k-free$ case in comparison with the squarefree case \cite{aymone_beyond_onehalf}} The proof of the main results in the $k$-free case follows closely the strategy in the squarefree case \cite{aymone_beyond_onehalf}. Briefly speaking, we can take a quadratic character $\chi$ (and modify it as in Theorem \ref{theorem 1}) and then we consider $f=\mu_k^2\chi$, where $\mu_k^2(n)$ is the indicator that $n$ is $k$-free. The Dirichlet series of $f$ is $L(s,\chi)/\zeta(ks)$ if $k$ is even, and it is $L(s,\chi)/L(ks,\chi)$ if $k$ is odd. Then we proceed with the Dirichlet Hyperbola method to upper bound the partial sums of $f$. The key difference, is that when $k$ is odd, we need to input estimates for $\sum_{n\leq x}\mu(n)\chi(n)$, and when $k$ is even, we need estimates for the partial sums of $\mu$. In our unconditional result, we need the estimates \eqref{equacao somas parciais da mobius incondicional} and \eqref{equacao somas parciais da mobius chi incondicional}, and in our conditional result, we need \eqref{equacao somas parciais da mobius condicional} and \eqref{equacao somas parciais da mobius chi condicional}. This explains why we need the Generalized Riemann hypothesis (for Dirichlet L-functions). Moreover, plugging the best conditional estimates for the partial sums of $\mu$ such as Soundararajan's upper bound mentioned in the introduction, perhaps we could substitute the term $x^\epsilon$ by a lower order term in the $k$-free and even case. But we prefer to leave it as it is.

\subsection{Proof of Theorem \ref{theorem 1}}  The next two Lemmas are just an adaptation of Lemmas 2.1 and 2.2 of \cite{aymone_beyond_onehalf} to our setting. In a few words, we show that if we modify a real and non-principal Dirichlet character $\chi$ in at most $\ll x^{1/k}\exp(-c\sqrt{\log x})$ primes $p\leq x$, then the resulting completely multiplicative function $g:\NN\to\{-1,1\}$ has partial sums $\sum_{n\leq x}g(n)\ll x^{1/k}\exp(-\delta\sqrt{\log x})$ for some $\delta>0$. The only difference with Lemmas 2.1 and 2.2 of \cite{aymone_beyond_onehalf} is the exponent $1/k$. Since the proofs are short we include it here for completeness.
\begin{lemma}\label{lemma 1} Let $h:\NN\to[0,\infty)$ be a multiplicative function such that:\\
(i) $h(p)\leq2$ and $h(p^r)\leq h(p)$, for all primes $p$ and all powers $r\geq 2$;\label{(i)}\\
(ii) For some constant $c>0$, $\sum_{p\leq x}h(p)\ll \frac{x^{1/k}}{\exp(c\sqrt{\log x})}$.\\
Then there exists a $\delta>0$ such that 
$$\sum_{n\leq x}h(n)\ll \frac{x^{1/k}}{\exp(\delta\sqrt{\log x})}.$$
\end{lemma}

\begin{proof} By partial summation, or by Kroenecker's Lemma (see \cite{shiryaev} page 390), we can obtain the stated conclusion if the following series converges for some $\delta>0$:
\begin{equation*}\sum_{n=1}^\infty\frac{h(n)\exp(\delta\sqrt{\log n})}{n^{1/k}}.
\end{equation*}

Notation: By $p^r\| n$ we mean that $r$ is the largest power of $p$ that divides $n$. Since $\sqrt{\log n}= \sqrt{\sum_{p^r\|n} \log p^r } \leq \sum_{p^r\| n} \sqrt{ \log p^r}$, we have that
\begin{equation*}\sum_{n\leq x}\frac{h(n)\exp(\delta\sqrt{\log n})}{n^{1/k}}\leq\sum_{n\leq x}\frac{\tilde{h}(n)}{n^{1/k}},
\end{equation*}
where $\tilde{h}$ is the multiplicative function such that $\tilde{h}(p^r)=\exp(\delta\sqrt{\log p^r})h(p^r)$, for all primes $p$ and all powers $r\geq 1$. We will show that the series $\sum_{n\leq x}\frac{\tilde{h}(n)}{n^{1/k}}$ converges. For that, by the Euler product formula (see Theorem 1.3, page 188 of \cite{tenenbaumlivro}), we only need to show that $\sum_{p\in\mathcal{P}}\sum_{r=1}^\infty\frac{\tilde{h}(p^r)}{p^{r/k}}$ converges, where $\mathcal{P}$ stands for the set of primes.

Let $0<\delta<c/2$ be small such that $\frac{\exp(\delta \sqrt{\log p})}{p^{1/k}}<1$ for all primes $p\in\mathcal{P}$. By the hypothesis \textit{(i)}
\begin{align*}
    \sum_{r=2}^\infty\frac{\tilde{h}(p^r)}{p^{r/k}}&=\sum_{r=2}^\infty\frac{h(p^r)\exp(\delta\sqrt{\log{p^r}})}{p^{r/k}}\leq h(p)\sum_{r=0}^\infty\frac{\exp(\delta\sqrt{r+2}\sqrt{\log{p}})}{p^\frac{r+2}{k}}\\
    &\leq h(p)\sum_{r=0}^\infty\frac{\exp(\delta (r+2)\sqrt{\log{p}})}{p^\frac{r+2}{k}}
    = h(p)\frac{\exp(2\delta \sqrt{\log{p}})}{p^{2/k}}\sum_{r=0}^\infty\frac{\exp(r\delta\sqrt{\log{p}})}{p^{r/k}}\\
    &=h(p)\frac{\exp(2\delta \sqrt{\log p})}{p^{2/k}}\frac{1}{1-\frac{\exp(\delta \sqrt{\log p})}{p^{1/k}}}\ll h(p)\frac{\exp(2\delta\sqrt{\log{p}})}{p^{1/k}(p^{1/k}-\exp(\delta\sqrt{\log{p}}))}\\
    &\ll_{\delta} \frac{h(p)\exp(2\delta \sqrt{\log p})}{p^{1/k}}.
\end{align*}

Define $T(x)=\sum_{p\leq x}h(p)$ if $x\geq1$. By expressing a sum by a Riemann-Stieltjes integral and then using the integration by parts formula:
\begin{align*}
    \sum_{p\leq x} \frac{h(p)\exp(2\delta\sqrt{\log p})}{p^{1/k}}&=\int_{1}^x \frac{\exp(2\delta\sqrt{\log t})}{t^{1/k}} dT(t)\\
    &\ll T(x)\frac{\exp(2\delta\sqrt{\log x})}{x^{1/k}}+\int_{1}^x T(t)\frac{\exp(2\delta\sqrt{\log t})}{t^{1+\frac{1}{k}}}dt\\
    &\ll \frac{1}{\exp((c-2\delta)\sqrt{\log x})}+\int_{1}^x \frac{1}{t\exp((c-2\delta)\sqrt{\log t})}dt\\
    &\ll 1,
\end{align*}
where in the penultimate inequality we used the hypothesis \textit{(ii)}.

Thus, 
\begin{equation*}
    \sum_{p\leq x}\sum_{r=1}^\infty\frac{\tilde{h}(p^r)}{p^{r/k}}\ll_\delta \sum_{p\leq x}\frac{h(p)\exp(2\delta\sqrt{\log{p}})}{p^{1/k}}\ll 1.
\end{equation*}
\end{proof}

\begin{lemma}\label{lemma 2} Let $g:\mathbb{N}\to\{-1,1\}$ be a completely multiplicative function. Assume that for some real and non-principal Dirichlet character $\chi$,  $g$ satisfies \eqref{condicao 1}:  
$$ \sum_{p\leq x}|1-g(p)\chi(p)|\ll \frac{x^{1/k}}{\exp(c\sqrt{\log x})}.$$
Then, for some $\delta>0$, 
$$\sum_{n\leq x}g(n)\ll \frac{x^{1/k}}{\exp(\delta\sqrt{\log x})}.$$
\end{lemma}

\begin{proof} Let $h=g\ast \chi^{-1}$, where $\chi^{-1}$ is the inverse of $\chi$ with respect to the Dirichlet convolution. Since $\chi$ is completely multiplicative, we have that $\chi^{-1}=\mu \chi$ (Theorem 2.17 of \cite{apostol}). Thus, $\chi^{-1}$ is multiplicative and has support on the squarefree integers. Therefore, for each prime $p$ and any power $r$:
\begin{align*}
|h(p^r)|&=|g\ast\chi^{-1}(p^r)|=\bigg|\sum_{d\mid p^r}g(d)\chi^{-1}\bigg(\frac{p^r}{d}\bigg)\bigg|=|g(p^r)+g(p^{r-1})\chi^{-1}(p)|\\
&=|g(p^r)||1+g(p)^{-1}\chi^{-1}(p)|\\
&=|1-g(p)\chi(p)|.
\end{align*}
Hence, $|h|$ satisfies the hypothesis of Lemma \ref{lemma 1}. Since $g=h\ast \chi$, we have that $\sum_{n\leq x}g(n)=\sum_{n\leq x}h(n) \sum_{m\leq x/n}\chi(m)$, and as $\chi$ has bounded partial sums, we have that $\sum_{n\leq x}g(n)\ll\sum_{n\leq x}|h(n)|$. \end{proof}

Before we state our next Lemma, we will introduce some notation. Given an arithmetic function $f$, we denote its partial sums up to $x$ by
$$M_f(x):=\sum_{n\leq x}f(n).$$
\begin{lemma}\label{Lemma 3}
Let $g:\NN\to\{-1,1\}$ be multiplicative and let $\chi$ be a real and non-principal Dirichlet character. Assume that for some $\delta>0$
\begin{equation*}
    \sum_{p\leq x}|1-g(p)\chi(p)|\ll \frac{x^{1/k}}{\exp(\delta\sqrt{\log x})},
\end{equation*}
where $k\geq 2$. Then, there exists a constant $c>0$ such that
\begin{equation*}
M_{\mu g}(x) \ll \frac{x}{\exp(c\sqrt{\log x})}.
\end{equation*}
\end{lemma}

\begin{proof} Recall that $\chi=(\mu\chi)^{-1}$, where the exponent $-1$ means that we are taking the inverse with respect to the Dirichlet convolution. Define $h:=\mu g\ast (\mu \chi)^{-1}$. Let $1/L(s,\chi)$, $Z(s)$, $H(s)$ be the Dirichlet series associated to $\mu \chi$, $\mu g$ and $h$, respectively. Recall the Euler product formula for $L(s,\chi)$:
\begin{equation*}
L(s,\chi)=\sum_{n=1}^\infty \frac{\chi(n)}{n^s}=\prod_{p\in \mathcal{P}}\bigg(1+\frac{\chi(p)}{p^s}+\frac{\chi(p)^2}{p^{2s}}+\cdots\bigg).
\end{equation*}
For $\mu g$ we have
\begin{equation*}
Z(s)=\sum_{n=1}^\infty \frac{\mu(n)g(n)}{n^s}=\prod_{p\in\mathcal{P}}\bigg(1-\frac{g(p)}{p^s}\bigg).
\end{equation*}

By the rules of Dirichlet convolution, we can write $H(s)=Z(s)L(s,\chi)$, and hence
\begin{align*}
    H(s)&=Z(s)L(s,\chi)=\prod_{p\in \mathcal{P}}\bigg(1+\frac{\chi(p)}{p^s}+\frac{\chi(p)^2}{p^{2s}}+\cdots\bigg)\bigg(1-\frac{g(p)}{p^s}\bigg)\\
    &=\prod_{p\in \mathcal{P}}\left( 1+\frac{\chi(p)-g(p)}{p^{s}}+\cdots+\frac{\chi(p)^{n-1}(\chi(p)-g(p))}{p^{ns}}+\cdots \right). 
\end{align*}
Since $h$ is multiplicative, we have that $h(p^n)=\chi(p)^{n-1}(\chi(p)-g(p))$. Thus $|h(p)| \leq 2$ and for all powers $r$, $|h(p^r)| \leq |h(p)|$. Moreover, since $|\chi(p) - g(p)| \leq |1-\chi(p) g(p)|$, we have by hypothesis that
\begin{align*}
    \sum_{p \leq x}|h(p)| &= \sum_{p \leq x}|\chi(p) - g(p)| \leq \sum_{p \leq x}|1-\chi(p)g(p)| \ll \frac{x^{\frac{1}{k}}}{\exp \left( \delta \sqrt{\log x} \right)}.
\end{align*}
Thus, by Lemma \ref{lemma 1}, there exists a new $\delta>0$ such that
\begin{equation*}
   | M_h(x)|\leq M_{|h|}(x) \ll \frac{x^{\frac{1}{k}}}{\exp \left( \delta \sqrt{\log x} \right)} \ll x^{1/k}.
\end{equation*}
In particular, by partial summation, the series $\sum_{n=1}^\infty\frac{|h(n)|}{n}$ converges.

Now we will use the Dirichlet Hyperbola method for $\mu g=h\ast \mu\chi$, see subsection \ref{subsection dirichlet hyperbola}. For all $U\geq 1$ and $V\geq 1$ such that $UV=x$, we have that
\begin{equation*}\label{metodo hiperbole}
    M_{\mu g}(x)=\sum_{n\leq V}h(n)M_{\mu \chi} \bigg( \frac{x}{n} \bigg) +\sum_{n\leq U}\mu(n)\chi(n)M_h\bigg{(}\frac{x}{n}\bigg{)}-M_{\mu \chi}(U)M_h(V):=A+B-C.
\end{equation*}
By \eqref{equacao somas parciais da mobius chi incondicional}, we have that for some constant $a>0$, $M_{\mu\chi}(x)\ll \frac{x}{\exp(a \sqrt{\log{x}})}$. Now choose $V=\exp(\epsilon\sqrt{\log{x}})$, where $0<\epsilon<a$. Further, let $c>0$ be a parameter $c<\min\left(\epsilon \frac{k-1}{k},a\right)$.

\noindent\textit{Estimate for $A$.}
\begin{align*}
|A|&\leq \sum_{n\leq V} | h(n)M_{\mu \chi} \bigg{(} \frac{x}{n} \bigg{)} | \ll 
\sum_{n\leq V} |h(n)| \frac{x}{n \exp \left( a\sqrt{\log\frac{x}{n}} \right)}\\
 &\ll \frac{x}{\exp \left( a\sqrt{\log\frac{x}{V}}\right) }  \sum_{n \leq V} \frac{|h(n)|}{n} \ll \frac{x}{\exp \left( a \sqrt{\log\frac{x}{V}}\right)} \\
   &\ll \frac{x}{\exp\left( a\sqrt{\log x-\epsilon\sqrt{\log x}}\right)} \ll \frac{x}{\exp \left( a\sqrt{\log x} \left( 1 - \frac{\epsilon}{\sqrt{\log x}}\right)^\frac{1}{2} \right)}\\
  &\ll \frac{x}{\exp \left( c \sqrt{\log x}\right) }.
\end{align*}
\noindent\textit{Estimate for $B$.}
\begin{align*}
    \left|B\right|&\leq \sum_{n\leq U}\left|\mu(n)\chi(n)M_h\bigg(\frac{x}{n}\bigg)\right|\ll\sum_{n\leq U}\bigg(\frac{x}{n}\bigg)^{1/k} \ll x^{1/k}\sum_{n\leq U}\frac{1}{n^{1/k}}\ll x^{\frac{1}{k}} U^{1-\frac{1}{k}}\\
    &\ll \frac{x}{\exp\left( \epsilon\frac{k-1}{k}\sqrt{\log{x}}\right)}\ll \frac{x}{\exp(c\sqrt{\log {x}})}.
\end{align*}
\noindent\textit{Estimate for $C$.}
\begin{align*}
    C=M_{\mu\chi}(U)M_h(V)&\ll \frac{U V^{1/k} }{\exp(a\sqrt{\log{U}})}\ll \frac{x}{V} \frac{V^\frac{1}{k}}{\exp(a\sqrt{\log{x}-\log{V})}}\\
    &\ll \frac{x}{V^{1-\frac{1}{k}}\exp\left( a\sqrt{\log{x}-\epsilon\sqrt{\log{x}}}\right)}\\
    &\ll \frac{x}{\exp\left( a\sqrt{\log{x}-\epsilon\sqrt{\log{x}}}+\epsilon\frac{k-1}{k}\sqrt{\log{x}}\right)}\\
    &\ll \frac{x}{\exp\left(\sqrt{\log{x}} \left( a\left( 1-\frac{\epsilon}{\sqrt{\log{x}}}\right)^{1/2}+\epsilon\frac{k-1}{k}\right)\right)}\\
    &\ll \frac{x}{\exp(c\sqrt{\log{x}})}.
\end{align*}
By combining these estimates for $A$, $B$, and $C$ we conclude the proof. \end{proof}
\begin{proof}[Proof of Theorem \ref{theorem 1}] Let $h:=f\ast g^{-1}$, where $g^{-1}$ is the inverse of $g$ with respect to the Dirichlet convolution. Let $F$, $G$ and $H$ be the Dirichlet series associated to $f$, $g$ and $h$ respectively. Since $g$ is completely multiplicative, by the Euler product formula we have that
\begin{equation*}
    G(s)=\prod_{p\in \mathcal{P}}\bigg(1-\frac{g(p)}{p^s}\bigg)^{-1}.
\end{equation*}
For $f=\mu_{k}^{2}g$, we have that
\begin{equation*}
    F(s)=\prod_{p\in\mathcal{P}}\bigg(1+\frac{f(p)}{p^s}+\frac{f(p^2)}{p^{2s}}+\dots\bigg)=\prod_{p\in\mathcal{P}}\bigg(1+\frac{g(p)}{p^s}+\frac{g(p)^2}{p^{2s}}+\dots+\frac{g(p)^{k-1}}{p^{(k-1)s}}\bigg).
\end{equation*}
Since $h=f\ast g^{-1}$, we can write $H(s)=F(s)/G(s)$, and hence
\begin{align*}
    H(s)&=\frac{F(s)}{G(s)}=\prod_{p\in\mathcal{P}}\bigg(1+\frac{g(p)}{p^s}+\frac{g(p)^2}{p^{2s}}+\dots+\frac{g(p)^{k-1}}{p^{(k-1)s}}\bigg)\bigg(1-\frac{g(p)}{p^s}\bigg)\\
    &=\prod_{p\in\mathcal{P}}\bigg(1-\frac{g(p)^k}{p^{ks}}\bigg).
\end{align*}
Since $h$ is multiplicative, we obtain that  $h(p^k)=-g(p)^k$ and $h(p^r)=0$, for all  $r\geq 1$ and $r\neq k$. Thus, $h(n)=0$ unless $n=m^k$ for some integer $m\geq 1$. Let $\ind_{\NN}(n^{1/k})$ be the indicator that $n^{1/k}$ is an integer. Now observe that if $k$ is odd, then $h(p^k)=-g(p)^k=-g(p)$, since $g(p)=\pm1$, and hence $h(n)=\ind_{\NN}(n^{1/k})\mu(n^{1/k})g(n^{1/k})$.

Thus, in the case that $k$ is odd, by Lemma \ref{Lemma 3}, for some constant $c>0$ we have that 
\begin{equation*}\label{equation partial sums of h with odd k}
M_h(x)=\sum_{n\leq x}h(n)=
\sum_{n\leq x}\ind_{\NN}(n^\frac{1}{k})\mu(n^\frac{1}{k})g(n^\frac{1}{k})\\
=\sum_{n\leq x^\frac{1}{k}}\mu(n)g(n)\\
\ll\frac{x^\frac{1}{k}}{\exp(c\sqrt{\log{x}^{1/k}})}.
\end{equation*}

If $k$ is even, then $h(p^k)=-1$, and hence $h(n)=\ind_{\NN}(n^{1/k})\mu(n^{1/k})$. By \eqref{equacao somas parciais da mobius incondicional}, we have for some (new) constant $c>0$, $M_\mu(x)\ll \frac{x}{\exp(c\sqrt{\log{x}})}$. Therefore,
\begin{equation*}\label{equation partial sums of h with even k}
    M_h(x)=\sum_{n\leq x}\ind_{\NN}(n^{\frac{1}{k}})\mu(n^{\frac{1}{k}})=\sum_{n\leq x^{\frac{1}{k}}}\mu(n)\ll\frac{x^{1/k}}{\exp(c\sqrt{\log{x^{1/k}}})}.
\end{equation*}
Thus, for any integer $k\geq 1$, for some constant $c>0$,
\begin{equation*}\label{equation partial sums of h}
    M_h(x)\ll\frac{x^{1/k}}{\exp(c\sqrt{\log{x^{1/k}}})}.\\
\end{equation*}
On the other hand, by Lemma \ref{lemma 2}, for some $\delta>0$, we have that
$$M_g(x)\ll \frac{x^{1/k}}{\exp(\delta\sqrt{\log x})}.$$
Now we use the Dirichlet hyperbola method (subsection \ref{subsection dirichlet hyperbola}): For all $U\geq 1$ and $V\geq 1$ such that $UV=x$, we have that
\begin{equation}\label{equation dirichlet hyperbola}
    M_f(x)=\sum_{n\leq U}h(n)M_g\bigg{(}\frac{x}{n}\bigg{)}+\sum_{n\leq V}g(n)M_h\bigg{(}\frac{x}{n}\bigg{)}-M_g(V)M_h(U):=A+B-C.
\end{equation}
We choose $V=\exp(\epsilon(\sqrt{\log x}))$, where $0<\epsilon<\frac{c}{\sqrt{k}}$ and $U=\frac{x}{V}$. Besides that, $\lambda>0$ is a parameter
$\lambda<\min(\delta\sqrt{\epsilon},\frac{c}{\sqrt{k}}-\frac{\epsilon(k-1)}{k})$.\\
 \textit{Estimate for $A$.}
\begin{align*}
    |A| &\leq \sum_{n\leq U}| \ind_{\NN}(n^\frac{1}{k})M_g(x/n)|=\sum_{n\leq U^\frac{1}{k}}| M_g(x/n^{k})|\\
    &\ll \sum_{n\leq U^\frac{1}{k}}\frac{x^\frac{1}{k}}{n}\frac{1}{\exp\bigg(\delta\sqrt{\log\frac{x}{n^k}}\bigg)}\ll \frac{x^\frac{1}{k}}{\exp\bigg(\delta\sqrt{\log\frac{x}{U}}\bigg)}\sum_{n\leq U^\frac{1}{k}}\frac{1}{n}\\
    &\ll \frac{x^\frac{1}{k}\log{U}}{\exp(\delta\sqrt{\log {V}})}\ll\frac{x^\frac{1}{k}(\log{x}-\log{V})}{\exp(\delta\sqrt{\epsilon}(\log{x})^\frac{1}{4})}\ll\frac{x^\frac{1}{k}\exp(\log{\log{x}})}{\exp(\delta\sqrt{\epsilon}(\log{x})^\frac{1}{4})}\\
    &\ll \frac{x^\frac{1}{k}}{\exp(\lambda(\log{x})^\frac{1}{4})},
\end{align*}
since $\lambda<\delta\sqrt{\epsilon}$.\\
\noindent  \textit{Estimate for $B$.} 
\begin{align*}
    |B|&= \sum_{n\leq V}|g(n)M_h(x/n)|\leq\sum_{n\leq V}|M_h(x/n)|\ll\sum_{n\leq V}\bigg(\frac{x}{n}\bigg)^\frac{1}{k}\frac{1}{\exp\bigg( c\sqrt{\log({\frac{x}{n})^\frac{1}{k}}}\bigg)}\\
    &\ll x^\frac{1}{k}\sum_{n\leq V}\frac{1}{n^\frac{1}{k}}\frac{1}{\exp\bigg(\frac{c}{\sqrt{k}}\sqrt{\log{\frac{x}{n}}}\bigg)}\ll \frac{x^\frac{1}{k}}{\exp\bigg(\frac{c}{\sqrt{k}}\sqrt{\log{\frac{x}{V}}}\bigg)}\sum_{n\leq V}\frac{1}{n^\frac{1}{k}}\\
    &\ll\frac{x^\frac{1}{k}V^\frac{k-1}{k}}{\exp\bigg( \frac{c}{\sqrt{k}}\sqrt{\log{x}-\log{V}}\bigg)} \ll \frac{x^\frac{1}{k}\exp\bigg(\epsilon\frac{k-1}{k}\sqrt{\log{x}}\bigg)}{\exp\bigg(\frac{c}{\sqrt{k}}\sqrt{\log{x}-\epsilon(\log{x})^\frac{1}{2}}\bigg)}\\
    &\ll \frac{x^\frac{1}{k}}{\exp\bigg(\frac{c}{\sqrt{k}}\sqrt{\log{x}-\epsilon(\log{x})^\frac{1}{2}}-\epsilon\frac{k-1}{k}\sqrt{\log{x}}\bigg)}\\
    &\ll \frac{x^\frac{1}{k}}{\exp\bigg((\log{x})^\frac{1}{2}\bigg(\frac{c}{\sqrt{k}}\bigg( 1-\frac{\epsilon}{ (\log{x})^{1/2}}\bigg)^\frac{1}{2}-\frac{\epsilon(k-1)}{k}\bigg)\bigg)} \ll \frac{x^\frac{1}{k}}{\exp\bigg(\lambda(\log{x})^\frac{1}{2}\bigg)}\\
    &\ll \frac{x^\frac{1}{k}}{\exp\big(\lambda(\log{x})^\frac{1}{4}\big)},
\end{align*}
since $0<\lambda<\frac{c}{\sqrt{k}}-\frac{\epsilon(k-1)}{k}$.

\noindent  \textit{Estimate for $C$.}
\begin{align*}
    C=M_g(V)M_h(U)&\ll \frac{V^\frac{1}{k}}{\exp(\delta\sqrt{\log{V}})}\frac{U^\frac{1}{k}}{\exp(c\sqrt{\log{U^{1/k}}})}\ll \frac{x^\frac{1}{k}}{\exp(\delta\sqrt{\log{V}})}\\
    &\ll \frac{x^\frac{1}{k}}{\exp(\delta\sqrt{\log{\exp(\epsilon\sqrt{\log{x}})}})}\ll \frac{x^\frac{1}{k}}{\exp(\delta\sqrt{\epsilon}(\log{x})^\frac{1}{4})}\\
    &\ll\frac{x^\frac{1}{k}}{\exp(\lambda(\log{x})^\frac{1}{4})}. 
\end{align*}
Combining these estimates for $A$, $B$ and $C$ we complete the proof. \end{proof}
\subsection{Proof of Theorem \ref{teorema riemann}}

\begin{proof}
Let $\chi$ be any real and non-principal Dirichlet character $\mbox{mod } q$ and $g:\NN\to\{-1,1\}$ be the completely multiplicative function such that:
\begin{align*}
g(n)&= \chi(n),\mbox{ if } \gcd(n,q)=1,\\
g(p)&=1, \mbox{ for each prime }p|q.
\end{align*}
Let $\omega(q)$ be number of distinct primes that divide $q$. Then, by Lemma 2.4 of \cite{aymone_beyond_onehalf} :
\begin{equation*}
\limsup_{x\to\infty} \frac{|M_g(x)|}{(\log x)^{\omega(q)}}\leq \frac{\max_{y\geq 1} |M_\chi(y)|}{\omega(q)!}\prod_{p|q}\frac{1}{\log p}.
\end{equation*}
In particular,  $M_g(x) \ll x^\alpha$, for each $\alpha>0$.

Let $h=f\ast g^{-1}$, where $g^{-1}$ is the Dirichlet inverse of $g$. As in the proof of theorems \ref{teorema 1} and \ref{theorem 1} above, if $k$ is odd, we have that  $h(n)=\ind_\NN(n^\frac{1}{k})\mu(n^\frac{1}{k})g(n^\frac{1}{k})$; If $k$ is even, then $h(n)=\ind_{\NN}(n^{\frac{1}{k}})\mu(n^{\frac{1}{k}})$.

In any case, by assuming GRH, we have \eqref{equacao somas parciais da mobius condicional} and \eqref{equacao somas parciais da mobius chi condicional}, and hence, for any $k\geq 2$,  $M_h(x) \ll x^{\frac{1}{2k}+\epsilon}$, for all $\epsilon>0$.

Now we use the Dirichlet Hyperbola method following the same notation of (\ref{equation dirichlet hyperbola}), and setting $U=x^{\frac{2k}{2k+1}}$ and $V=x^{\frac{1}{2k+1}}$. It is noteworthy that these choices for $U$ and $V$ are optimal in our method.

\noindent  \textit{Estimate for $A$.}
\begin{align*}
    |A|&\leq\sum_{n\leq U}\bigg|h(n)M_g\bigg(\frac{x}{n}\bigg)\bigg|=\sum_{n\leq U^\frac{1}{k}}\bigg|M_g\bigg(\frac{x}{n^k}\bigg)\bigg|\ll\sum_{n\leq U^\frac{1}{k}}\frac{x^\alpha}{n^{k\alpha}}\ll x^{\alpha}U^{\frac{1}{k}(1-k\alpha)}\\
    &\ll x^{\alpha}x^{\frac{2}{2k+1}(1-k\alpha)}\ll x^{\frac{2}{2k+1}+\alpha\left( 1-\frac{2k}{2k+1}\right)} \ll x^{\frac{2}{2k+1}+\frac{\alpha}{2k+1}}.
\end{align*}

\noindent  \textit{Estimate for $B$.}
\begin{align*}
    |B|&\leq\sum_{n\leq V}\bigg|g(n)M_h\bigg(\frac{x}{n}\bigg)\bigg|\ll x^{\frac{1}{2k}+\epsilon}\sum_{n\leq V}\frac{1}{n^{\frac{1}{2k}+\epsilon}}\ll x^{\frac{1}{2k}+\epsilon}V^{1-\frac{1}{2k}-\epsilon}\\
    &\ll x^{\frac{1}{2k}+\epsilon}x^{\frac{1}{2k+1}\left( 1-\frac{1}{2k}-\epsilon\right)}\ll x^{\frac{2}{2k+1}+\epsilon\left( 1-\frac{1}{2k+1}\right)}\ll x^{\frac{2}{2k+1}+\frac{\epsilon 2k}{2k+1}}.
\end{align*}

\noindent  \textit{Estimate for $C$.}
\begin{align*}
    C=M_g(V)M_h(U)\ll V^{\alpha}U^{\frac{1}{2k}+\epsilon}\ll x^{\frac{\alpha}{2k+1}}x^{\frac{1}{2k+1}+\frac{\epsilon 2k}{2k+1}}\ll x^{\frac{1}{2k+1}+\frac{\alpha}{2k+1}+\frac{\epsilon 2k}{2k+1}}.
\end{align*}
Since we have freedom to choose $\alpha$ and $\epsilon$ arbitrarily small, the proof is complete. \end{proof}

\section{A simulation}

In the figure below, we plot the partial sums of $\mu^2\chi_3$ up to $x=10^7$, where $\chi_3$ is the non-principal Dirichlet character with modulus $3$.

\begin{figure}[h]
\includegraphics[scale=0.5]{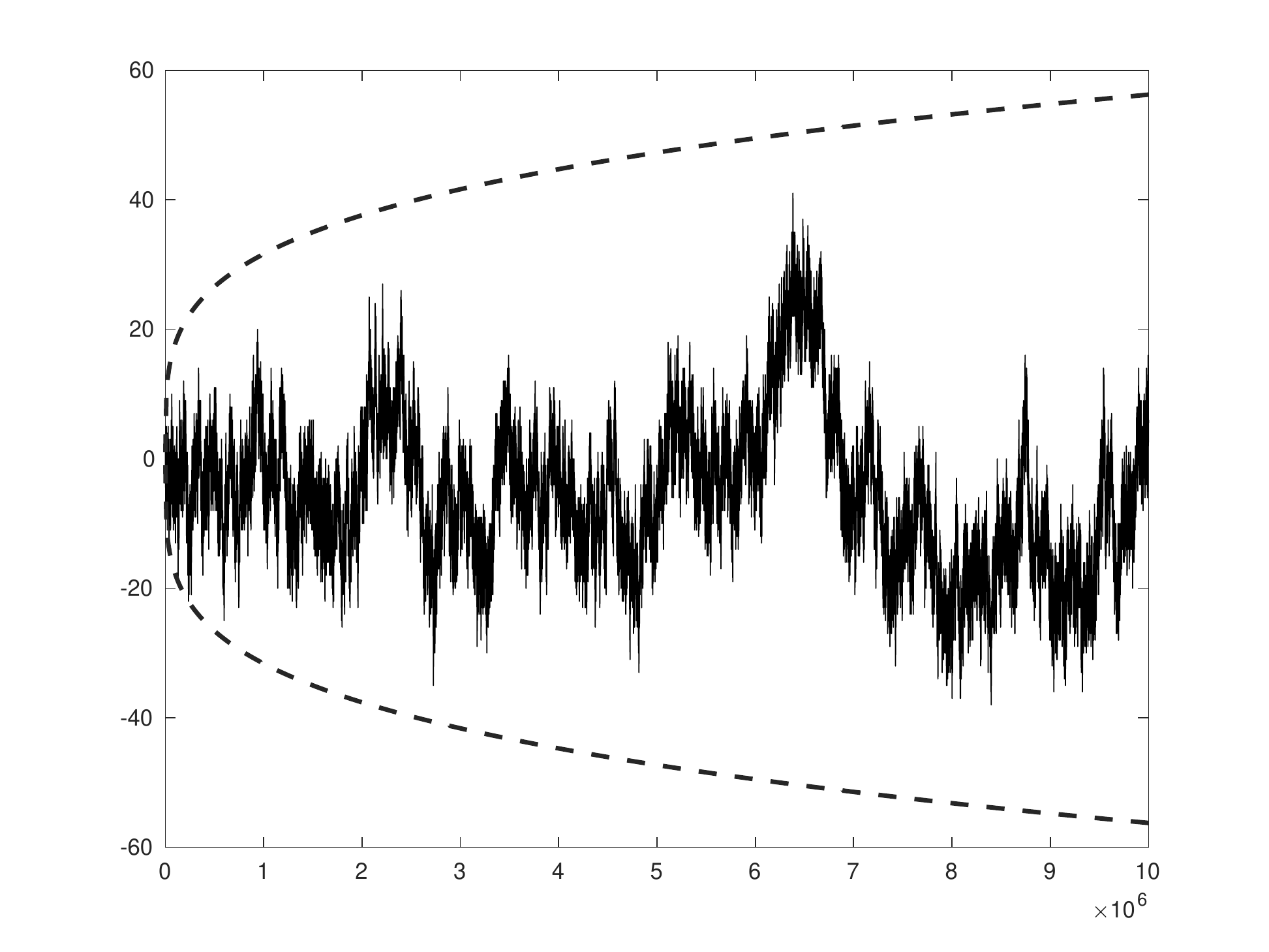} 
\caption{The dashed curves are given by $x\mapsto\pm x^{1/4}$, and the continuous line is given by $x\mapsto\sum_{n\leq x}\mu^2(n)\chi_3(n)$.}
\end{figure}

As we observed in the introduction, for a quadratic character $\chi$, $\sum_{n\leq x}\mu_k^2(n)\chi(n)$ is not $O(x^{1/2k-\epsilon})$, and the figure above is in agreement with this in the case $k=2$. Our conditional result says that the partial sums up to $x$ are $\ll x^{1/(k+1/2)+\epsilon}$ and as the figure suggests in the case $k=2$, we conjecture that actually is $\ll x^{1/2k+\epsilon}$, for all $\epsilon>0$.

\noindent\textbf{Acknowledgements.} We would like to thank the anonymous referee for a careful reading of the paper and for useful suggestions that improved the exposition. MA is supported by CNPq - grant Universal number 403037/2021-2; CB was financed in part by the scholarship from Coordenação de Aperfeiçoamento de Pessoal de Nível Superior - Brasil (CAPES) - Finance code 001; KM was financed by the Fapemig scholarship.

\noindent {\small{\sc
 Departamento de Matem\'atica, Universidade Federal de Minas Gerais (UFMG), Av. Ant\^onio Carlos, 6627, Belo Horizonte, Minas Gerais, Caixa Postal 702, CEP 31270-901, Brazil} \\
\textit{Email address:} \\
aymone.marco@gmail.com \\
 caiomafiabueno@gmail.com \\
  ktmedeiros36@gmail.com}

\end{document}